\documentclass[10pt]{article}
\usepackage[english]{babel}
\usepackage[T1]{fontenc}
\usepackage[utf8]{inputenc}

\usepackage{graphics,graphicx}	
\usepackage{hyperref, breakcites}
\usepackage{enumitem, mathrsfs}  	
\usepackage{amsmath, amsthm, amsfonts}

\usepackage[export]{adjustbox}	

\usepackage{url}
\usepackage[colorinlistoftodos]{todonotes}

\newtheorem{Thm}{Theorem}
\newtheorem{Lem}{Lemma}
\newtheorem{Def}{Definition}

\newtheorem{Prop}{Proposition}

\theoremstyle{remark}

\theoremstyle{definition}

\title{\textbf{On the Martingale Representation with Respect to the super-Brownian Filtration}}

\author{Christian Mandler\thanks{Mathematisches Institut, Justus-Liebig-Universit\"{a}t Giessen (\textit{e-mail: \href{mailto:christian-mandler@math.uni-giessen.de}{Christian.Mandler@math.uni-giessen.de}}) Financial support from the Deutsche Forschungsgemeinschaft is gratefully acknowledged. }
\and Ludger Overbeck\thanks{Mathematisches Institut, Justus-Liebig-Universit\"{a}t Giessen (\textit{e-mail: \href{mailto:Ludger.Overbeck@math.uni-giessen.de}{Ludger.Overbeck@math.uni-giessen.de}}) }
}

\begin{document}

\maketitle 
\thispagestyle{empty}

\begin{abstract}
We derive the explicit form of the martingale representation for square-integrable processes that are martingales with respect to the natural filtration of the super-Brownian motion. This is done by using a weak extension of the Dupire derivative for functionals of superprocesses.
\end{abstract}

\section*{Introduction}

Dupire's landmark work on the functional It\={o}-formula \cite{Du09} gave rise to a completely new approach to numerous questions in the field of stochastic calculus. In this paper, we use it to derive a martingale representation of functionals of the super-Brownian motion, a well-studied, infinite-dimensional, namely measure-valued, stochastic process. \\

Cont and Fournié (\cite{CoFo10}, \cite{CoFo13}, \cite{Co16}) as well as Levental et al. (\cite{Lev13}) used different formalizations of Dupire's idea to derive versions of a functional It\={o}-formula for real-valued semi-martingales using different derivatives. In addition, Cont used the functional It\={o}-formula to derive a martingale representation formula for Brownian martingales and extended the functional derivative to a weak derivative for square-integrable functionals (see \cite{CoFo13}, \cite{Co16}). \\

In \cite{MaOv20} the functional It\={o}-formula for $B(A,c)$-super\-processes is derived using derivatives extending the ones introduced by Cont and Fournié to functionals on the infinite-dimensional space of measures. In this paper, we extend the vertical derivative to a weak one and use it to derive a martingale representation for functionals of the super-Brownian motion.  \\

The martingale representation of functionals of superprocesses has been studied in \cite{Ov95}, \cite{EP94} and \cite{EP95}, where the authors proved the existence of a martingale representation and derive the explicit form for the richer class of so-called historical processes using their cluster representation, from which one can derive the representation for superprocesses as a projection. The approach presented in Section \ref{sec.Mart} is based on the ideas presented by Cont and Fournié in \cite{CoFo13} and \cite{Co16}, does not use the concept of historical processes and only considers the super-Brownian motion as a prototype of more general measure-valued diffusions.

\section{The setting}

Consider a filtered probability space $(\Omega, \mathcal{F}, (\mathcal{F}_t)_{t \in [0,T]}, \mathbb{P})$, where $\mathbb{P}$ is the law of the super-Brownian motion on $\mathbb{R}^d$, $d \geq 1$, starting at $m \in M_F(\mathbb{R}^d)$, $\mathcal{F}$ is the Borel-$\sigma$-field corresponding to $\Omega= C([0,T], M_F(\mathbb{R}^d))$ and $(\mathcal{F}_t)_{t \in [0,T]}$ is the canonical filtration satisfying the usual conditions. Also, denote the Laplacian operator by $\Delta$. From \cite{Daw93}, we know that the probability measure $\mathbb{P}$ is a solution to the following martingale problem:
\begin{gather*}
\mathbb{P}(X(0) = m) = 1 \text{ and for all $\phi \in D(\textstyle{\frac{1}{2}}\Delta)$ the process} \\
M(t)(\phi) = \langle X(t), \phi \rangle - \langle X(0) , \phi \rangle - \int_0^t \langle X(s), \frac{1}{2}\Delta \phi \rangle ds , \quad t \in [0,T] \\
\text{is a $(\mathcal{F}_t)_t$-martingale $\mathbb{P}$-a.s. and has quadratic variation}\\
[M(\phi)]_t = \int_0^t \langle X(s), \phi ^2 \rangle ds \quad \forall t \in [0,T] \quad \mathbb{P}-a.s..
\end{gather*}
As the super-Brownian motion is a $B(A,c)$-superprocess, the process $M(\phi)$ in the martingale problem gives rise to a martingale measure in the sense of \cite{Wal86} which we denote by $M_X$. This martingale measure is defined on $\Omega \times \mathcal{B}[0,T] \times \mathcal{E}$ and plays in a crucial role in the It\={o}-formulas presented in \cite{MaOv20} as well as the results introduced below. \\

Denote by $\mathbb{S}(\mathbb{R}^d)$ the Schwartz space on $\mathbb{R}^d$. Functions in $\mathbb{S}(\mathbb{R}^d)$ are infinitely continuously differentiable. Thus $\mathbb{S}(\mathbb{R}^d) \subset D(\frac{1}{2} \Delta)$ and for any $h \in \mathbb{S}(\mathbb{R}^d)$ it holds $\frac{1}{2} \Delta h \in D(\frac{1}{2}\Delta)$.\\

In addition, let $\mathcal{S}$ be the space of simple functions, i.e. the space of functions that are linear combinations of functions $f: \Omega \times [0,T] \times  \mathbb{R}^d \rightarrow \mathbb{R}$ of form
\begin{equation*}
f(\omega, s, x) = X(\omega)  1_{(a,b]}(s) 1_A(x)
\end{equation*}
with $0 \leq a < b \leq t$, $X$ bounded and $\mathcal{F}_a$-measurable and $A \in \mathcal{B}( \mathbb{R}^d)$. Denote by $\mathcal{P}$ the predictable $\sigma$-field on $\Omega \times [0,T] \times  \mathbb{R}^d$, i.e. the $\sigma$-field generated by $\mathcal{S}$. If a function is $\mathcal{P}$-measurable, we say it is predictable. \\

In Section \ref{sec.Mart} we consider functions $F$ of the path of the super-Brownian motion $X$. More precisely, we consider the stopped path $X_t$ of the process $X$. To formalize the notion of (stopped) paths, consider an arbitrary path $\omega \in D([0,T], M_F( \mathbb{R}^d))$, the space of right continuous functions with left limits. We equip $D([0,T], M_F( \mathbb{R}^d))$ with the metric $\tilde{d}$ given by 
\begin{equation*}
\tilde{d}(\omega, \omega') = \sup_{u \in [0,T]} d_P( \omega (u), \omega'(u))
\end{equation*}
for $\omega$, $\omega' \in D([0,T], M_F( \mathbb{R}^d))$, where $d_P$ is the Prokhorov metric on $M_F( \mathbb{R}^d)$. For such a path, define the path stopped at time $t$ by $\omega_t(u) = \omega ( t \wedge u)$. 
Using this, we define an equivalence relation on the space $[0,T] \times D([0,T], M_F( \mathbb{R}^d))$ by 
\begin{equation*}
(t ,  \omega) \sim (t' , \omega') \quad \Leftrightarrow \quad t = t' \ \text{and} \ \omega_t = \omega'_{t'} .
\end{equation*} 
This relation gives rise to the quotient space 
\begin{align*}
\Lambda_T &= \{ (t , \omega_t) \ | \ (t, \omega) \in [0,T] \times D([0,T], M_F( \mathbb{R}^d)) \} \\
&= ([0,T] \times D([0,T], M_F( \mathbb{R}^d))) / \sim .
\end{align*}
Next, define a metric $d_\infty$ on $\Lambda_T$ by
\begin{equation*}
d_\infty((t, \omega), (t', \omega')) = \sup_{u \in [0,T]} d_P (\omega(u \wedge t), \omega'(u \wedge t')) + | t - t'|.
\end{equation*}

A functional $F : \Lambda_T \rightarrow \mathbb{R}$ is continuous with respect to $d_\infty$ if for all $(t, \omega) \in \Lambda_T$ and every $\varepsilon > 0$ there exists an $\eta > 0$ such that for all $(t', \omega') \in \Lambda_T$ with $d_\infty((t,\omega), (t', \omega')) < \eta$ we have
\begin{equation*}
|F(t, \omega) - F(t' , \omega')| < \varepsilon .
\end{equation*}

Further, a functional $F$ on $[0,T] \times D([0,T], M_F( \mathbb{R}^d))$ is called \emph{non-anticipative} if it is a measurable map on the space of stopped paths, i.e. $F: \Lambda_T \rightarrow \mathbb{R}$. In other words, $F$ is non-anticipative if $F (t, \omega) = F(t, \omega_t)$ holds for all $\omega \in D([0,T], M_F( \mathbb{R}^d))$. \\

For continuous non-anticipative functionals, we can define the following derivative. For this, denote by $\delta_x$, $x \in \mathbb{R}^d$, the Dirac measure with unit mass at $x$.

\begin{Def}
A continuous non-anticipative functional $F : \Lambda_T \rightarrow \mathbb{R}$ is called \emph{vertically differentiable} at $(t, \omega) \in \Lambda_T$ in direction $\delta_x$, $x \in \mathbb{R}^d$, if the limit
\begin{equation*}
\mathcal{D}_x F(t, \omega) = \lim_{\varepsilon \rightarrow 0} \frac{F(t, \omega_t + \varepsilon \delta_x 1_{[t,T]}) - F(t, \omega_t)}{\varepsilon}
\end{equation*}
exists. If this is the case for all $(t, \omega) \in \Lambda_T$, we call $\mathcal{D}_xF$ the \emph{vertical derivative of $F$ in direction $\delta_x$}. Higher order vertical derivatives are defined iteratively.
\end{Def}

\section{The Explicit Form of the Martingale Representation for the Super-Brownian Motion}\label{sec.Mart}

Let $\mathcal{L}^2 (M_X)$ be the space of predictable functions $\phi : \Omega \times [0,T] \times  \mathbb{R}^d \rightarrow \mathbb{R}$  such that\footnote{This space coincides with the space $\mathcal{P}_M$ in \cite{Wal86}. Therefore the stochastic integral with respect to the martingale measure $M_X$ is well defined for all $\phi \in \mathcal{L}^2(M_X)$.}
\begin{equation*}
\| \phi \|_{\mathcal{L}^2(M_X)}^2 = \mathbb{E} \left[ \int_0^T \int_{\mathbb{R}^d} \phi^2(s,x) X(s) (dx)ds \right] < \infty ,
\end{equation*}
$\mathcal{M}^2$ be the space of square-integrable $(\mathcal{F}_t)_t$-martingales with initial value zero and with norm 
\begin{equation*}
\| Y \|^2_{\mathcal{M}^2} = \mathbb{E} [ Y(T)^2 ] .
\end{equation*}

Further define the space $U$ which is the linear span of functions of form
\begin{equation*}
\phi_{\Gamma,a,h}(\omega, t,x) = \Gamma {(\omega)} \cdot h(x) 1_{(a ,T]}(t) 
\end{equation*}
with $\Gamma$ bounded, $a \in [0,T)$ and $\mathcal{F}_a$-measurable and $h \in \mathbb{S}(\mathbb{R}^d)$.  As functions in $U$ can be expressed as pointwise limits of functions in $\mathcal{S}$, these functions are predictable. 

\begin{Prop} \label{AppendC}
The space $U$ is a dense subset of $\mathcal{L}^2(M_X)$.
\end{Prop}

\begin{proof}
From the above, we already know that functions $\phi \in U$ are predictable. Hence to show that $U$ is indeed a subspace of $\mathcal{L}^2(M_X)$, we have to prove that $\| \phi \|^2 < \infty$ holds for all $\phi \in U$. \\

It is enough to show that $\| \phi_{\Gamma, a, h} \|_{\mathcal{L}^2(M_X)}^2 < \infty$ for all functions $ \phi_{\Gamma, a, h}$ generating $U$. Assuming that the bounds of $\Gamma^2$ and $h^2$ are given by $C_{\Gamma^2}$ and $C_{h^2}$, respectively, we get 
\begin{align*}
\| \phi_{\Gamma, a, h} \|_{\mathcal{L}^2(M_X)}^2 &= \mathbb{E}\left[  \int_0^T \int_{\mathbb{R}^d} (\Gamma h(x) 1_{(a , T]}(s))^2 X(s)(dx) ds \right] \\
&\leq  C_{\Gamma^2} \mathbb{E}\left[ \int_0^T \int_{\mathbb{R}^d}  h^2(x) 1_{(a , T]}(s) X(s)(dx) ds \right] \\
&\leq  C_{\Gamma^2} \, C_{h^2}  \mathbb{E}\left[ \int_a^T  X(s)(\mathbb{R}^d) ds \right] \\
&\leq  C_{\Gamma^2} \, C_{h^2} \, (T-a) \, \mathbb{E}\left[ \max_{t \in [a, T]}  X(t)(\mathbb{R}^d) \right] .
\end{align*}

As the total mass of $X$ is a critical Feller continuous state branching process, we know that 
\begin{equation*}
\mathbb{E}[ X(t)(\mathbb{R}^d)] = X(0)(\mathbb{R}^d) < \infty,
\end{equation*}
holds for all $t \in [0,T]$, which yields that $\phi_{\Gamma, a , h}$ has finite $\| \cdot \| _{\mathcal{L}^2(M_X)}$-norm. Consequently $\phi \in \mathcal{L}^2(M_X)$ and thus $U \subset \mathcal{L}^2(M_X)$. \\

To obtain the density of $U$ in $\mathcal{L}^2(M_X)$, note that $\mathcal{S}$ is dense in $\mathcal{L}^2(M_X)$, i.e. $\bar{\mathcal{S}} = \mathcal{L}^2(M_X)$ holds. Since the closures of $\mathcal{S}$ and $U$ coincide, we obtain $\bar{U} = \bar{\mathcal{S}} = \mathcal{L}^2(M_X)$ and thus that $U$ is dense in $\mathcal{L}^2(M_X)$.
\end{proof}

In the following, we derive the martingale representation formula for all square-integrable $(\mathcal{F}_t)_t$-martingales. We do so by extending the notion of vertical derivatives to obtain an operator $\nabla_M$ on $\mathcal{M}^2$ which plays a crucial role in the martingale representation. To get the general representation formula, we start by deriving the formula and defining the operator $\nabla_M$ for martingales in a subset of $\mathcal{M}^2$. As the considered subspace is dense in $\mathcal{M}^2$, we can then extend the operator as well as the martingale representation formula to all $Y \in \mathcal{M}^2$ and thus obtain a martingale representation for all square-integrable $(\mathcal{F}_t)_t$-martingales. 

\begin{Def}
A linear operator $\Pi$ mapping from its domain $D(\Pi)$ into a Hilbert space $\mathcal{H}$ is called an \emph{extension} of the linear operator $\tilde{\Pi} : D(\tilde{\Pi}) \rightarrow \mathcal{H}$ if $D(\tilde{\Pi}) \subset D(\Pi)$ and $\tilde{\Pi} v = \Pi v$ for all $v \in D(\tilde{\Pi})$. 
\end{Def}

To construct the just mentioned subspace of $\mathcal{M}^2$, we consider the stochastic integral $I_{M_X}(\phi)$ of a function $\phi \in \mathcal{L}^2(M_X)$ with respect to the martingale measure $M_X$ corresponding to $X$. This allows us to define the set
\begin{equation*}
I_{M_X}(U) = \{ Y \ | \ Y(t) = \int_0^t \int_{\mathbb{R}^d} \phi(s,x) M_X(ds, dx) , \ \phi \in U , \ t \in [0,T] \} .
\end{equation*}
To prove that this is indeed a subset of $\mathcal{M}^2$, we need parts of the proof of the following result.

\begin{Prop} \label{PropIso}
The mapping 
\begin{align}
\begin{split}
I_{M_X}: \quad  \mathcal{L}^2(M_X) \quad & \rightarrow \quad \mathcal{M}^2 \\
\phi \quad  & \mapsto \quad  \int_0^\cdot \int_{\mathbb{R}^d} \phi(s,x) M_X(ds, dx) 
\end{split}
\label{DefIM}
\end{align}
is an isometry.
\end{Prop}

\begin{proof}
Let $Q_{M_X}$ be the covariation of $M_X$ and $\phi$, $\psi \in \mathcal{L}^2(M_X)$. Then, by adapting the proof of Theorem 2.5 in \cite{Wal86} to our setting, we get that 
\begin{align*}
&\int_0^t \int_{\mathbb{R}^d} \phi (s,x) M_X(ds, dx) \int_0^t \int_{\mathbb{R}^d} \psi(s,x) M_X(ds,dx)  \\
-\quad  & \int_0^t \int_{\mathbb{R}^d} \int_{\mathbb{R}^d} \phi(s,x) \psi(s,y) Q_{M_X}(ds, dx, dy)
\end{align*}
is a martingale for all $\phi$, $\psi \in \mathcal{L}^2(M_X)$ and thus 
\begin{align*}
&\mathbb{E} \left[\int_0^t \int_{\mathbb{R}^d} \phi (s,x) M_X(ds, dx) \int_0^t \int_{\mathbb{R}^d} \psi(s,x) M_X(ds,dx)  \right] \\
=\quad  & \mathbb{E} \left[ \int_0^t \int_{\mathbb{R}^d} \int_{\mathbb{R}^d} \phi(s,x) \psi(s,y) Q_{M_X}(ds, dx, dy) \right].
\end{align*}
In the scenario we consider $M_X$ is orthogonal. For orthogonal martingale measures, the covariation $Q_{M_X}$ simplifies to $Q([0,t], B, B ) = \nu([0,t], B)$ with $\nu$ being a measure on $\mathbb{R}^d \times [0,T]$. From Example 7.1.3 in \cite{Daw93} we know that $\nu$ has the following form if $M_X$ is the martingale measure associated with the $B(A,c)$-superprocess:
\begin{equation*}
\nu( ds, dx) = c X(s)(dx) ds.
\end{equation*}

As $c=1$ in our setting, this yields
\begin{align}
\begin{split}
&\mathbb{E} \left[\int_0^t \int_{\mathbb{R}^d} \phi (s,x) M_X(ds, dx) \int_0^t \int_{\mathbb{R}^d} \psi(s,x) M_X(ds,dx) \right] \\
=\quad  & \mathbb{E} \left[  \int_0^t \int_{\mathbb{R}^d} \phi(s,x) \psi(s,x) X(s)(dx) ds \right]
\end{split}
\label{eq1}
\end{align}
for all $t \in [0,T]$ and therefore
\begin{align*}
\| I_{M_X}(\phi) \|_{\mathcal{M}^2}^2 &= \mathbb{E} \left[  \left( \int_0^T \int_{\mathbb{R}^d} \phi(s,x) M_X(ds, dx) \right) ^2 \right] \\
&= \mathbb{E} \left[ \int_0^T \int_{\mathbb{R}^d} \phi^2 (s,x) X(s) (dx)ds \right] \\
&= \| \phi \|_{\mathcal{L}^2(M_X)}^2 .
\end{align*}
\end{proof}

For any function in $\phi = \phi_{\Gamma, a, h} \in U$ we have
\begin{align*}
I_{M_X}(\phi)(t) &= \int_0^t \int_{\mathbb{R}^d} \phi_{\Gamma,a,h}(s,x) M_X(ds, dx) \\
&= \int_0^t \int_{\mathbb{R}^d} \Gamma \cdot h(x) 1_{(a ,T]}(s) M_X(ds, dx) \\
&= \Gamma \cdot ( M(t)(h) - M(a)(h)) 1_{t >a}\\
&= \Gamma \cdot \left( \vphantom{\int_0^2}\langle X(t), h\rangle - \langle X(a) , h \rangle - \int_{a}^t \int_{\mathbb{R}^d} \frac{1}{2} \Delta h(y) X(s)(dy) ds\right) 1_{t > a}.
\end{align*}
From the martingale problem defining the super-Brownian motion, we get that $(M(t)(h))_{t \in [0,T]}$ is a martingale as $h \in \mathbb{S}(\mathbb{R}^d) \subset D(\frac{1}{2} \Delta)$ and thus $(I_{M_X}(\phi)(t))_{t \in [0,T]}$ is also a martingale for any $\phi \in U$. The process is also square-integrable as we have for all $t \in [0,T]$
\begin{align*}
\mathbb{E} [(I_{M_X}(\phi)(t))^2] &= \mathbb{E}\left[ \left( \int_0^t \int_{\mathbb{R}^d} \phi (s,x) M_X(ds, dx) \right)^2 \right] \\
&= \mathbb{E} \left[  \int_0^t \int_{\mathbb{R}^d} \phi^2(s,x) X(s)(dx)ds \right] \\ 
& \leq \mathbb{E} \left[  \int_0^T \int_{\mathbb{R}^d} \phi^2(s,x) X(s)(dx)ds \right] \\
& < \infty, 
\end{align*}
since $\phi \in \mathcal{L}^2(M_X)$. Hence $I_{M_X}(\phi)$ is square-integrable and thus $I_{M_X}(U) \subset \mathcal{M}^2$. \\

Next, consider the function $F$ of form
\begin{align*}
F :\quad \Lambda_T \quad &\rightarrow  \quad \mathbb{R} \\
(t,\omega) \quad &\mapsto \quad \Gamma{(\omega)}  \left( \vphantom{\int_0^2}\langle \omega(t), h\rangle - \langle \omega(a) , h \rangle - \int_{a}^t \int_{\mathbb{R}^d} \frac{1}{2} \Delta h(y)  \omega(s)(dy) ds\right) 1_{t > a} .
\end{align*}
Plugging $X_t$ into $F$ for $\omega$ yields
\begin{align*}
F(t, X_t) &= \Gamma(X_t) \left( \langle X_t(t) , h \rangle - \langle X_t(a) , h \rangle - \int_a^t \int_{\mathbb{R}^d} \frac{1}{2}\Delta h(x) X_t(s)(dx)ds \right) 1_{t > a} \\
 &= \Gamma(X_t) \left( \langle X(t) , h \rangle - \langle X(a) , h \rangle - \int_a^t \int_{\mathbb{R}^d} \frac{1}{2}\Delta h(x) X(s)(dx)ds \right) 1_{t > a}
\end{align*}
and, as $X_t(\omega) = \omega_t$ and $X(t)(\omega) = \omega(t)$, we get
\begin{align*}
F(t, X_t)(\omega) = \Gamma(\omega_a) \left( \langle \omega(t) , h \rangle - \langle \omega(a), h) - \int_a^t \int_{\mathbb{R}^d} \langle \omega(s), \frac{1}{2} \Delta h(x) \omega(s)(dx) ds \right),
\end{align*}
from which we get, as $\Gamma$ is $\mathcal{F}_a$-measurable, that $F(t, X_t) = I_{M_X} (\phi_{\Gamma,a,h})(t)$. \\

As, in addition, for any path $\omega \in C([0,T], M_F(E))$, it holds 
\begin{align*}
&\lim_{\varepsilon \downarrow 0} \frac{1}{\varepsilon} \left( \vphantom{\int_0^1} \langle (\omega + \varepsilon \delta_x 1_{[t,T]})(t), h\rangle - \langle (\omega + \varepsilon \delta_x 1_{[t,T]})(a) , h \rangle \right.\\
& \hphantom{\lim_{\varepsilon \downarrow 0} \frac{1}{\varepsilon} \quad \langle (\omega + \varepsilon \delta_x 1_{[t,T]})(t), h\rangle} -\int_{a}^t \int_{\mathbb{R}^d} \frac{1}{2} \Delta h(y) (\omega + \varepsilon \delta_x 1_{[t,T]})(r)(dy) dr  \\
& \hphantom{\quad \lim_{\varepsilon \downarrow 0} \frac{1}{\varepsilon}} \left.-  \, \langle \omega(t), h\rangle + \langle \omega(a) , h \rangle +\int_{a}^t \int_{\mathbb{R}^d} \frac{1}{2} \Delta h(y) \omega(r)(dy) dr \right) \\
= \ & \lim_{\varepsilon \downarrow 0} \frac{1}{\varepsilon} \left( \varepsilon h(x) - \varepsilon \int_{a}^t \int_{\mathbb{R}^d} \frac{1}{2} \Delta h(y) 1_{[t,T]}(r) \delta_x(dy) dr \right) \\
= \ & h(x)
\end{align*}
and, as $\Gamma$ is $\mathcal{F}_a$-measurable, for $t \in (a, T]$ it holds
\begin{equation*}
\lim_{\varepsilon \rightarrow 0} \frac{1}{\varepsilon} \left( \Gamma(\omega + \varepsilon \delta_x 1_{[t,T]}) - \Gamma(\omega) \right) = 0,
\end{equation*}
we obtain, for all $(t, \omega) \in \Lambda_T$,
\begin{equation*} 
\mathcal{D}_x F(t, \omega) = \Gamma {(\omega)} 1_{(a ,T]}(t) h(x) = \phi_{\Gamma, a, h}(t,x).
\end{equation*}

Now, for a process $Y$ defined by 
\begin{equation*}
Y(t) = I_{M_X} (\phi_{\Gamma,a,h})(t) = F(t, X_t)
\end{equation*}
for $t \in [0,T]$, we can define an operator $\nabla_M$ of the form 
\begin{align}
\begin{split}
\nabla_M: \quad  I_{M_X}(U) \quad & \rightarrow \quad \mathcal{L}^2(M_X) \\
Y \quad & \mapsto \quad \nabla_M Y,
\end{split}
\label{DefNab}
\end{align}
where $\nabla_M Y$ is given by
\begin{equation*}
\nabla_M Y : \quad (\omega, t, x) \quad \mapsto \quad \nabla_M Y(\omega, t, x) := \mathcal{D}_x F (t, X_t(\omega)) = \mathcal{D}_x F (t, \theta) |_{\theta = X_t(\omega)} .
\end{equation*}
Further, by the definition of $Y$, we get the martingale representation 
\begin{equation}
Y(t) = \int_0^t \int_{\mathbb{R}^d} \nabla_M Y(s,x) M_X(ds, dx) ,
\label{mart}
\end{equation}
which is equal to 
\begin{equation*}
F(t, X_t) = \int_0^t \int_{\mathbb{R}^d} \mathcal{D}_x F(s, X_s) M_{X}(ds, dx).
\end{equation*}

Thus, we derived a martingale representation formula for martingales in the subspace $I_{M_X}(U)$ of $\mathcal{M}^2$ using the operator $\nabla_M$ defined by \eqref{DefNab} on $I_{M_X}(U)$. \\

In \cite{EP94} the authors proved that, if $Y$ belongs to $\mathcal{M}^2$, there exists a unique $\rho \in \mathcal{L}^2 (M_X)$ such that
\begin{equation}
Y(t) =  \int_0^t \int_{\mathbb{R}^d} \rho (s,x) M_X(ds, dx) \quad \forall t \geq 0 
\label{martRepGen}
\end{equation}
holds $\mathbb{P}$-a.s.. Consequently, the representation in \eqref{mart} is unique. Further, this yields that the mapping $I_{M_X}$ is bijective. These two properties allow us to extend $\nabla_M$ to all processes $Y \in \mathcal{M}^2$ by using the following result.

\begin{Prop}
The space $\{\nabla_M Y \ | \ Y \in I_{M_X}(U) \}$ is dense in $\mathcal{L}^2(M_X)$ and the space $I_{M_X}(U)$ is dense in $\mathcal{M}^2$.
\end{Prop}

\begin{proof}
As $U$ is dense in $\mathcal{L}^2(M_X)$ (see Proposition \ref{AppendC}),
\begin{equation*}
U = \{ \nabla_M Y \ | \  Y \in I_{M_X}(U) \} \subset \mathcal{L}^2 (M_X)
\end{equation*}
yields that $\{\nabla_M Y \ |  \  Y \in I_{M_X}(U) \}$ is dense in $\mathcal{L}^2(M_X)$. Further, as $I_{M_X}$ is a bijective isometry, we get the density of $I_{M_X}(U)$ in $\mathcal{M}^2$. \\
\end{proof}

This density result allows to prove the following proposition that can be interpreted as an integration by parts formula on $[0,T] \times \mathbb{R}^d \times \Omega$. 

\begin{Prop} \label{PropNabla}
If $Y \in I_{M_X}(U)$, then $\nabla_M Y$ is the unique element in $\mathcal{L}^2 (M_X)$ such that
\begin{equation}
\mathbb{E} \left[ Y(T) Z(T) \right] = \mathbb{E} \left[  \int_0^T \int_{\mathbb{R}^d} \nabla_M Y(s,x) \nabla_M Z(s,x) X(s)(dx) ds \right]
\label{eq3}
\end{equation}
holds for all $Z \in I_{M_X}(U)$. 
\end{Prop}

\begin{proof}
From \eqref{eq1} and \eqref{mart} we get that, for every $Y$, $Z \in I_{M_X}(U)$, 
\begin{align*}
\mathbb{E} [ Y(T)Z(T)] = \ & \mathbb{E} \left[ \int_0^T \int_{\mathbb{R}^d} \nabla_M Y(s,x) M_X(ds,dx) \int_0^T \int_{\mathbb{R}^d} \nabla_M Z(s,x) M_X(ds, dx) \right] \\
= \ & \mathbb{E} \left[  \int_0^T \int_{\mathbb{R}^d} \nabla_M Y(s,x) \nabla_M Z(s,x) X(s)(dx) ds \right] 
\end{align*}
holds. \\

To prove the uniqueness, let $\psi \in \mathcal{L}^2(M_X)$ be another process such that for $Z \in I_{M_X}(U)$ we have
\begin{equation*}
\mathbb{E} \left[ Y(T)Z(T) \right] \\
 = \mathbb{E} \left[  \int_0^T \int_{\mathbb{R}^d} \psi (s, x) \nabla_M Z(s,x) X(s)(dx)ds \right].
\end{equation*}
Then, by subtraction
\begin{equation*}
0 = \mathbb{E} \left[ \int_0^T \int_{\mathbb{R}^d} (\psi(s, x) - \nabla_M Y(s,x)) \nabla_M Z(s,x) X(s) (dx) ds \right]
\end{equation*}
for all $Z \in I_{M_X}(U)$, which yields the uniqueness of $\nabla_M Y$ in $\mathcal{L}^2(M_X)$ as $\{ \nabla_M Z \ | \ Z \in I_{M_X}(U)\}$ is dense in $\mathcal{L}^2(M_X)$. \\
\end{proof}

The interpretation as an integration by parts formula becomes clear when we consider the following alternative form of equation \eqref{eq3}, which holds for all $\phi \in \mathcal{L}^2(M_X)$:
\begin{equation*}
\mathbb{E} \left[ Y(T) \int_0^T \int_{\mathbb{R}^d} \phi(s,x) M_X(ds,dx) \right] = \mathbb{E} \left[  \int_0^T \int_{\mathbb{R}^d} \nabla_M Y(s,x) \phi(s,x) X(s)(dx) ds \right] .
\end{equation*}

Now, we have all the necessary results on hand to introduce the extension of $\nabla_M$ to $\mathcal{M}^2$.

\begin{Thm} \label{Thm.Ext}
The operator defined in \eqref{DefNab} can be extended to an operator 
\begin{align*}
\nabla_M: \quad \mathcal{M}^2 \quad & \rightarrow \quad \mathcal{L}^2(M_X) \\
Y \quad & \mapsto \quad \nabla Y .
\end{align*}
This operator is a bijection and the unique continuous extension of the operator defined in \eqref{DefNab} given by the following: For a given $Y \in \mathcal{M}^2$, $\nabla_M Y$ is the unique element in $\mathcal{L}^2(M_X)$ such that
\begin{equation}
\mathbb{E}[Y(T) Z(T)] = \mathbb{E} \left[ \int_0^T \int_{\mathbb{R}^d} \nabla_M Y(s,x) \nabla_M Z(s,x) X(s)(dx) ds \right]
\label{eq.Ext}
\end{equation}
holds for all $Z  \in I_{M_X}(U)$.
\end{Thm}

\begin{proof}
As $\nabla_M : I_{M_X}(U) \rightarrow \mathcal{L}^2(M_X)$ is a bounded linear operator, $\mathcal{L}^2(M_X)$ is a Hilbert space and $I_{M_X}(U)$ is dense in $\mathcal{M}^2$, the the BLT theorem (see e.g. Theorem 5.19 in \cite{HuNa01}) yields the existence of a unique continuous extension $\nabla_M : \mathcal{M}^2 \rightarrow \mathcal{L}^2(M_X)$. To prove that \eqref{eq.Ext} defines the extension, we have to prove that its restriction to $I_{M_X}(U)$ is equal to the initial operator. As we immediately get this from Proposition \ref{PropNabla}, the unique continuous extension is given by \eqref{eq.Ext}. \\

For $Y \in \mathcal{M}^2$, there exists a unique $\rho \in \mathcal{L}^2(M_X)$ such that \eqref{martRepGen} holds. This $\rho$ satisfies \eqref{eq.Ext} as we get from combining \eqref{martRepGen} with \eqref{mart} and \eqref{eq1} that 
\begin{align*}
\mathbb{E}[Y(T) Z(T)] &= \mathbb{E}\left[ \int_0^T \int_{\mathbb{R}^d} \rho (s,x) M_X(ds,dx) Z(T)\right] \\
&= \mathbb{E} \left[ \int_0^T \int_{\mathbb{R}^d} \rho(s,x) \nabla_M Z(s,x) X(s)(dx) ds \right]
\end{align*}
holds for all $Z \in I_{M_X}(U)$. The uniqueness of the integrand in \eqref{martRepGen} then yields that $\rho$ and $\nabla_M Y$ coincide in $\mathcal{L}^2(M_X)$. \\

Next, assume $Y$, $Y' \in \mathcal{M}^2$ with $\nabla_M Y = \nabla_M Y'$ for $\nabla_M Y$, $\nabla_M Y' \in \mathcal{L}^2(M_X)$. Then, for all $Z \in I_{M_X}(U)$,
\begin{align*}
0 &= \mathbb{E} \left[ \int_0^T \int_{\mathbb{R}^d} (\nabla_M Y(s,x) - \nabla_M Y'(s,x) ) \nabla_M Z(s,x) X(s)(dx) ds \right] \\
&= \mathbb{E} \left[  \left(\int_0^T \int_{\mathbb{R}^d} \nabla_M Y(s,x) M_X(ds,dx) ds  - \int_0^T \int_{\mathbb{R}^d} \nabla_M Y'(s,x) M_x(ds,dx)  \right)  Z(T) \right] \\
&= \mathbb{E}[ (Y(T) - Y'(T)) Z(T)],
\end{align*}
which implies that $Y = Y'$ in $\mathcal{M}^2$ as $I_{M_X}(U)$ is dense in $\mathcal{M}^2$. Therefore, the operator $\nabla_M$ is injective. Further, as for every $\phi \in \mathcal{L}^2(M_X)$ there exists the process
\begin{equation*}
Y =  \int_0^\cdot \int_{\mathbb{R}^d} \phi(s,x) M_X(ds, dx) \in \mathcal{M}^2
\end{equation*}
for which $\nabla_M Y = \phi$ holds, the operator is also surjective and thus bijective. 
\end{proof}

The extension of $\nabla_M$ from $I_{M_X}(U)$ to $\mathcal{M}^2$ now immediately yields the following martingale representation for square-integrable martingales. \\

\begin{Thm}
For any square-integrable $(\mathcal{F}_t)_t$-martingale $Y$ and every $t \in [0,T]$ it holds
\begin{equation*}
Y(t) = Y(0) + \int_0^t \int_{\mathbb{R}^d} \nabla_M Y(s,x) M_X(ds, dx) \quad \text{$\mathbb{P}$-a.s. .}
\label{martRep}
\end{equation*}
\end{Thm}

\begin{proof}
First, assume $Y \in \mathcal{M}^2$. Then there exists a unique $\rho \in \mathcal{L}^2(M_X)$ such that \eqref{martRepGen} holds and from the proof of Theorem \ref{Thm.Ext} we get that this $\rho$ satisfies \eqref{eq.Ext}. Therefore, by the uniqueness of $\rho$ and the definition of $\nabla_M Y$, we have $\rho = \nabla_M Y$ in $\mathcal{L}^2(M_X)$ and thus 
\begin{equation*}
Y(t) = \int_0^t \int_{\mathbb{R}^d} \nabla_M Y(s,x) M_X(ds,dx)
\end{equation*}
holds $\mathbb{P}$-a.s.. \\

Now, let $Y$ be a square-integrable $(\mathcal{F}_t)_t$-integral. Then $\tilde{Y} = Y - Y(0) \in \mathcal{M}^2$ and thus we can apply the above to $\tilde{Y}$. Adding $Y(0)$ to both sides of the resulting representation for $\tilde{Y}$ yields the desired martingale representation formula for $Y$. 
\end{proof}

We conclude this paper with the two following properties of the extended operator $\nabla_M$ defined on $\mathcal{M}^2$, which are worth mentioning.

\begin{Lem}
The operator $\nabla_M$ defined on $\mathcal{M}^2$ is an isometry and the adjoint operator of $I_{M_X}$, the stochastic integral with respect to the martingale measure $M_X$.
\end{Lem}

\begin{proof}
Let $Y \in \mathcal{M}^2$. Then we get, using the same arguments as in Proposition \ref{PropIso},
\begin{align*}
\| \nabla_M Y \|_{\mathcal{L}^2(M_X)}^2 &= \mathbb{E} \left[  \int_0^T \int_{\mathbb{R}^d} (\nabla_M Y(s,x))^2 X(s) (dx) ds \right] \\
&= \mathbb{E} \left[ \left( \int_0^T \int_{\mathbb{R}^d} \nabla_M Y(s,x) M_X(ds, dx) \right)^2 \right] \\
&= \left\| \int_0^\cdot \int_{\mathbb{R}^d} \nabla_M Y(s,x) M_X(ds, dx) \right\| _{\mathcal{M}^2}^2 \\
&= \| Y \|_{\mathcal{M}^2}^2 .
\end{align*}
Further, let $\phi \in \mathcal{L}^2(M_X)$. As
\begin{align*}
\langle I_{M_X}(\phi) , Y \rangle_{\mathcal{M}^2} &= \mathbb{E} \left[ \int_0^T \int_{\mathbb{R}^d} \phi (s,x) M_X(ds,dx) Y(T) \right] \\
&= \mathbb{E} \left[ \int_0^T \int_{\mathbb{R}^d} \phi(s,x) M_X(ds, dx) \int_0^T \int_{\mathbb{R}^d} \nabla_M Y(s,x) M_X(ds,dx) \right] \\
&= \mathbb{E} \left[  \int_0^T \int_{\mathbb{R}^d} \phi(s,x) \nabla_M Y(s,x) X(s)(dx)ds \right] \\
&= \langle \phi , \nabla_M Y \rangle_{\mathcal{L}^2(M_X)}
\end{align*}
holds, we get that $\nabla_M$ is the adjoint operator of $I_{M_X}$.
\end{proof}

\bibliography{Martingale_Representation_SBM}
\bibliographystyle{apalike}

\end{document}